\documentclass[12pt]{article}
\usepackage{amssymb,amsmath}
\newcommand{\qed}{\hfill \rule{2.5mm}{2.5mm}}
\newcommand{\R}{{\mathbb R}}

\newcommand{\N}{{\mathbb N}}

\newcommand{\proof}{{\em Proof:\ }}

\begin{document}
\newtheorem{thm}{Theorem}[section]
\newtheorem{defs}[thm]{Definition}
\newtheorem{lem}[thm]{Lemma}
\newtheorem{rem}[thm]{Remark}
\newtheorem{cor}[thm]{Corollary}
\newtheorem{prop}[thm]{Proposition}
\renewcommand{\theequation}{\arabic{section}.\arabic{equation}}
\newcommand{\newsection}[1]{\setcounter{equation}{0} \section{#1}}
\title{Mixing inequalities in Riesz spaces
      }
\author{
{Wen-Chi Kuo}\footnote{Supported in part by NRF grant number CSUR160503163733.}\\
School of Computational and Applied Mathematics\\
University of the Witwatersrand\\
Private Bag 3, P O WITS 2050, South Africa\\ \\
{Michael J. Rogans}\\
School of Statistics and Actuarial Science\\
University of the Witwatersrand\\
Private Bag 3, P O WITS 2050, South Africa\\ \\
{ Bruce A. Watson} \footnote{Supported in part by the Centre for Applicable Analysis and
Number Theory and by NRF grant number IFR2011032400120.}\\
School of Mathematics\\
University of the Witwatersrand\\
Private Bag 3, P O WITS 2050, South Africa}
\maketitle
\abstract{\noindent
Various topics in stochastic processes have been considered in the abstract setting of Riesz spaces, for example martingales, martingale convergence, ergodic theory, AMARTS, Markov processes and mixingales. Here we continue the relaxation of conditional independence begun in the study of mixingales and study mixing processes. The two mixing coefficients which will be considered are
the $\alpha$ (strong) and $\varphi$ (uniform) mixing coefficients.  We conclude with mixing
inequalities for these types of processes. In order to facilitate this development, the study
of generalized $L^1$ and $L^\infty$ spaces begun by Kuo, Labuschagne and Watson will be 
extended.
}
\parindent=0in
\parskip=.2in
\newsection{Introduction}

Mixing processes are stochastic processes in which independence assumptions are replaced by a measure of independence called the mixing coefficient, see \cite{bill, mix1, edgar-s, mix3} for measure theoretic essentials of mixing processes.  
In the Riesz space (measure free) setting, processes which require independence, such as Markov 
processes, were considered with independence replaced by conditional independence, see
\cite{markov}.
In line with the above approach, mixingales (processes with independence/conditional independence in the limit) were
considered in the Riesz space setting in \cite{mix}.
In this work we will pose $\alpha$ (strong) and $\varphi$ (uniform) mixing processes in the Riesz space setting.
Core to the theory of mixing is the family of inequalities generally referred to as the mixing inequality, the conditional Riesz space analogues of  which form the focus of this paper, see Section~\ref{section-inequality}. We will give mixing inequalities for both $\alpha$ and $\varphi$
mixing (the one being an easy consequence of the other). We refer the reader to
\cite{bill, edgar-s, ibragimov, mix3, serfling}
 for the
measure, non-conditional analogues.
To facilitate this study the Riesz space analogues of the $L^p$ spaces introduced in \cite{stochint} will be revisited.

In \cite{conditional}, it was shown that a conditional expectation operator, $T$, on a Riesz space, $E$, admits a unique maximal extension to a conditional expectation operator, also denoted $T$, in the universal completion, $E^u$, of $E$, with domain a Dedekind complete Riesz space, which will be denoted $L^1(T)$. The procedure used there was based on that of de Pagter and Grobler, \cite{ multiplication operators}, for the measure theoretic setting. We observe here that the range of the maximal extension of the conditional expectation operator, i.e. $R(T):=\{Tf\,|\,f\in L^1(T)\}$, is a universally complete $f$-algebra and that $L^1(T)$ is an $R(T)$-module. This prompts the definition of an $R(T)$ (vector valued) norm $\|\cdot\|_{T,1}:=T|\cdot|$ on $L^1(T)$. 
Here the homogeneity is with respect to multiplication by elements of $R(T)_+$.
Following in a similar manner $L^\infty(T)$ is taken to
be the subspace of $L^1(T)$ composed of $R(T)$ bounded elements. An $R(T)$ valued norm $\|\cdot\|_{T,\infty}:=\inf\{g\in R(T)_+\,|\,|\cdot|\le g\},$ 
is defined on $L^\infty(T)$.
This extends on the concepts of $L^\infty(T)$ defined in  \cite{stochint}.
 
 In \cite[Sections 39, 42 and 43]{D-M-B}
 Dellacherie and Meyer gave an extension of martingale theory to $\sigma$-finite processes.
As a direct application of the material presented in Sections 2 to 4, we give, in Section 5, an extension of
the theory of mixing theory to $\sigma$-finite processes. The extension of mixing theory even to this special case, to the knowledge of the authors,
 has not been considered in the literature. 
 
 Natural connections with the theory presented here are to laws of large numbers and other convergence theorems.  For the conditionally independent case in Riesz spaces we refer the reader to Stoica \cite{stoica}. 
 
\section{Preliminaries}

For general background on Riesz spaces we refer the reader to \cite{zaanen} and for the foundations
of stochastic processes in Riesz spaces to \cite{discrete, conditional, convergence, mix, stochint}. We will present only the essential results from the theory of Riesz space stochastic processes required for our consideration of mixing processes.

\begin{defs}
Let $E$ be a Dedekind complete Riesz space with weak order unit. A positive order continuous linear projection $T$ on $E$ with range $R(T)$, a Dedekind complete Riesz subspace of $E$, is said to be a conditional expectation operator
if $Te$ is a weak order unit of $E$ for each weak order unit $e$ of $E$.
\end{defs}

A Riesz space $E$ is said to be universally complete
if $E$ is Dedekind complete and every subset of $E$ which consists of mutually disjoint elements has a supremum in $E$.
A Riesz space $E^{u}$ is said to be a universal completion
of the Riesz space $E$ if $E^{u}$ is universally complete and contains $E$ as an order dense subspace. 
If $e$ is a weak order unit of $E$ then $e$ is also a weak order unit of $E^{u}$, see \cite{zaanen}.

We say that a conditional expectation operator, $T$, on a Riesz space is strictly positive if $T|f|=0$ implies that $f=0$.
As shown in \cite{conditional}, a strictly positive conditional expectation operator, $T$, on a  Riesz space can be extended to its so called natural domain (maximal domain to which it can be extended as a conditional expectation operator) denoted $\mbox{dom}(T)$. We set $L^1(T):=\mbox{dom}(T)$ and we denote the extension of $T$ to $L^1(T)$ again by $T$. This is consistent with the special case of $T$ an expectation operator in the measure theoretic setting, see \cite{multiplication operators}. 
In particular, if $E$ is a Dedekind complete
Riesz space with weak order unit and conditional expectation operator, $T$, we say that $E$ is $T$-universally complete if $E=L^1(T)$. From the definition of $\mbox{dom}(T)$, $E$ is 
$T$-universally complete if and only
 if for each upwards directed net $(f_{\alpha})_{\alpha \in \Lambda}$ in $E_{+}$ such that $(Tf_{\alpha})_{\alpha \in \Lambda}$ is order bounded in $E^{u}$, we have that $(f_{\alpha})_{\alpha \in \Lambda}$ is order convergent in $E$. 

If $e$ is a weak order unit of $E$ then we denote the $f$-algebra of $e$ bounded elements by $$E^e:=\{f\in E\,|\, |f|\le ke,\, \mbox{ for some }\, k\in \R_+\}$$
and set
$$L^\infty(T):=\{f\in L^1(T)\,|\,|f|\le g, \mbox{ for some } g\in R(T)_+\}.$$
We recall from \cite[Theorem 5.3]{conditional} that each conditional expectation operator $T$ is an averaging operator in the sense that if $f\in R(T)$ 
and $g\in E$ with $fg\in E$ then $T(fg)=fT(g)$.

\begin{thm}
 \label{R(T) x R(T) = R(T)}
 Let $E$ be a $T$-universally complete Riesz space, where $T$ is a conditional expectation
  operator on 
  $E$, and let $e$ be a weak order unit for $E$ with $Te=e$.
  Then $R(T)$ is universally complete and hence an $f$-algebra. In addition $E=L^1(T)$ and $L^\infty(T)$ are $R(T)$-modules.
\end{thm}

\begin{proof}
In order to show that $R(T)$ is universally complete, we need to show that for each
$W\subset R(T)_+$ consisting of mutually disjoint elements, i.e. if $u,v\in W$ with $u\ne v$ then
$u\wedge v =0$, we have that $w:=\vee_{v\in W}v\in R(T)$. To this end, let $W$ be as above
and set $w:=\vee_{v\in W}v$, which exists in $E^u$.
Now $u\wedge ne\in R(T)_+$ and $u\wedge ne\le ne$, for all $n\in\N, u\in W$.
Here $\{u\wedge ne\,|\, u\in W\}\subset R(T)$ is bounded above by  $ne\in R(T)$ and so the Dedekind completeness of $R(T)$ allows us to conclude that 
$$w\wedge ne=\vee_{u\in W} (u\wedge ne) \in R(T).$$    
Now $w\wedge ne\uparrow w$ and 
$$T(w\wedge ne)=w\wedge ne\le w\in E^u, n\in \N,$$
so the $T$-universal completeness of $E$ gives that $w\in E$.
The order continuity of $T$ gives $Tw=w$ and $w\in R(T)$.
Thus $R(T)^u=R(T)$ from which it follows that $R(T)$ is an $f$-algebra.
 
Since  $fg=f^+g^++f^-g^--f^-g^+-f^+g^-$, to 
show that $L^1(T)$ is an $R(T)$-module, it suffices to show that
$fg\in L^1(T)$ for each $f\in L^1(T)_+$ and $g\in R(T)_+$.
Now from the averaging properties of $T$ and as $(f\wedge ne)g\in L^1(T)$ we have that
$$T((f\wedge ne)g)=T(f\wedge ne)g\le T(f)g\in E^u_+$$
for all $n\in\N$. However $L^1(T)$ is T-universally complete and 
$(f\wedge ne)g\uparrow fg$, so $fg\in L^1(T)$.
 
Finally, to show that $L^\infty(T)$ is an $R(T)$-module we, take
$f\in L^\infty(T)_+$ and $g\in R(T)_+$.
Here there is $F\in R(T)$ with  $0\le f\le F$. Now 
 $$fg\le Fg\in R(T)_+$$
 as $R(T)$ is an algebra. Hence $fg\in L^\infty(T)$.
  \qed
\end{proof}

We note the connection with the work of Grobler and de Pagter in \cite{multiplication operators}, where in the measure theoretic case, it is shown that the range space of the maximal extension of a classical conditional expectation operator is an algebra.

\section{The $L^p(T), p=1,\infty,$ spaces with $T$-norms}

In the previous section we define the $L^p(T), p=1,\infty,$ spaces, see also \cite{conditional, stochint}. Here we present the corresponding vector valued generalizations of the $L^p$-norms. In particular, these norms take their values in the positive cone of the universally complete algebra $R(T)$. 
We also refer the reader to the Riesz semi-norm approach used by Grobler and Labuschagne, \cite{GL-2}, to study the space $L^2(T)$.
For some recent progress in this area we refer the reader to \cite{C-V}.
  
\begin{defs}\label{def-norm}
 Let $E$ be a Dedekind complete Riesz space with weak order unit and $T$ be a strictly positive
  conditional expectation operator on $E$. If $E$ is an $R(T)$-module and
   $\phi : E\to R(T)_+$  with
   \begin{description}
    \item[(a)] $\phi(f)=0$ if and only if $f=0$,
    \item[(b)] $\phi(gf)=|g|\phi(f)$ for all $f\in E$ and $g\in R(T)$,
    \item[(c)] $\phi(f+h)\le \phi(f)+\phi(h)$ for all $f,h\in E$,
   \end{description}
 then $\phi$ will be called an $R(T)$-valued norm on $E$.
 \end{defs}

\begin{thm}
 \label{norm}
 Let $E$ be a $T$-universally complete Riesz space with weak order unit, 
  where $T$ is a strictly positive conditional expectation operator on $E$.
  The map $$f\mapsto T|f|=:\|f\|_{T,1}$$ defines an $R(T)$-valued norm on $L^1(T)$,  and the map
   $$f\mapsto\| f \|_{T,\infty} = \inf{\left\{ g \in R(T)_+ \, \middle\vert \, |f|  \le g \right\}}$$ 
   defines an $R(T)$-valued norm on $L^\infty(T)$.
   \end{thm} 

\begin{proof}
For $L^1(T)$, condition (a) of Definition~\ref{def-norm}  follows from the strict positivity of $T$ and while (c)  of Definition~\ref{def-norm}  follows from the linearity of $T$ and $|\cdot|$ obeying the triangle inequality. For (b) of Definition~\ref{def-norm}, let $f\in E$ and $g\in R(T)$, then we observe that the terms
$f^+g^+,f^-g^+,f^+g^-,f^-g^-$ are disjoint and positive. Thus
$$|gf|=|f^+g^+-f^-g^+-f^+g^-+f^-g^-|=f^+g^++f^-g^++f^+g^-+f^-g^-=|g||f|.$$
Here $|g|\in R(T)$, so $T|gf|=T(|g||f|)=|g|T|f|$, from which Definition~\ref{def-norm} part (b) follows.
 
We now consider the case of $L^\infty(T)$. Here (a) and (b) follow directly from the definition
of $\|\cdot\|_{T,\infty}$. For (c), consider $f,g\in L^\infty(T)$. As $|f+g|\le |f|+|g|$, it follows that
\begin{eqnarray*}
\{ h \in R(T)_+ \, | \, |f|+|g|  \le h \} \subset \{ h \in R(T)_+ \, | \, |f+g|  \le h \}.
\end{eqnarray*}
Therefore
\begin{eqnarray*}
\|f+g\|_{T,\infty} = \inf\{ h \in R(T)_+ \,| \, |f+g|\le h\} 
\le \inf\{ h \in R(T)_+ \, | \, |f|+|g|\le h\}.
\end{eqnarray*}
Writing $h = h_f + h_g$, for $h_f, h_g \in R(T)_+$, it follows that
\begin{eqnarray*}
\{h_f+h_g \,|\, |f|\le h_f, |g|\le h_g, h_f,h_g \in R(T)_+ \} \subset \{ h \in R(T)_+ \,|\, |f|+|g|\le h \},
\end{eqnarray*}
giving
\begin{eqnarray*}
\inf\{h_f+h_g \,|\, |f|\le h_f, |g|\le h_g, h_f,h_g \in R(T)_+\} \ge \inf\{ h \in R(T)_+ \,|\, |f|+|g|\le h \}.
\end{eqnarray*}
Combining the above and noting that the conditions $\left|f\right|  \le h_{f}$ and $\left|g\right|  \le h_{g}$ are independent, we have
\begin{eqnarray*}
\|f+g\|_{T,\infty} 
& \le&
\inf\{h_f+h_g \,|\, |f|\le h_f, |g|\le h_g, h_f,h_g \in R(T)_+\}\\
 &=&
\inf\{h_f \,|\, |f|\le h_f, h_f \in R(T)_+\} 
+\inf\{h_g \,|\, |g|\le h_g, h_g \in R(T)_+\}\\
&=& \|f \|_{T,\infty} + \|g\|_{T,\infty}.\quad \qed
\end{eqnarray*}
 \end{proof}

\begin{thm}
 \label{L1 x Linfty = L1}
If $f \in {L}^{1}(T)$ and $g \in {L}^{\infty}(T)$, then $gf \in {L}^{1}(T)$ and
$$\|gf\|_{T,1}\le \|g\|_{T,\infty}\|f\|_{T,1}.$$
\end{thm}

\begin{proof}
From linearity, it suffices to show that $gf \in {L}^{1}(T)$ for 
$f \in {L}^{1}(T)_{+}$ and $g \in {L}^{\infty}(T)_{+}$. 
Here there exists $h \in R(T)_{+}$ such that $g  \le h$, and we note from earlier that
$hf\in {L}^{1}(T)_{+}$, but $0\le gf\le hf$, giving $gf\in {L}^{1}(T)_{+}$.

Now for each $h\in R(T)_+$ with $|g|\le h$ we have from $T$ being an averaging operator that
$$\|gf\|_{T,1}=T|gf|=T(|g||f|)\le T(h|f|)=hT|f|=h\|f\|_{T,1}.$$
However, from the compatability of the multiplicative structure with the order structure,
$$\inf\{h\|f\|_{T,1}\,|\,|g|\le h,h\in R(T)_+\}=\inf\{h\in R(T)_+\,|\,|g|\le h\}\|f\|_{T,1}=\|f\|_{T,1}\|g\|_{T,\infty},$$
from which the result follows.
\qed
\end{proof}

Setting $f=e$, where $e$ is a weak order unit with $Te=e$ and which has been chosen as the algebraic unit for the $f$-algebra structure, we have the following corollary to 
Theorem \ref{L1 x Linfty = L1}.

\begin{cor} \label{baw-norm-inequality}
If $g \in {L}^{\infty}(T)$ then $\|g\|_{T,1}\le \|g\|_{T,\infty}.$
\end{cor}

To conclude this section, we give a variant of the conditional Jensen's inequality. For additional details on conditional Jensen's inequalities in Riesz spaces, see \cite{jensen's inequality}.

\begin{thm}
\label{jensen's inequality}
If $S$ is a conditional expectation
operator on $L^1(T)$ compatible $T$ (in the sense that $TS=T=ST$), then
\begin{eqnarray*}
\|Sf \|_{T, \, p}  \le \| f \|_{T,p},
\end{eqnarray*}
for all $f \in L^p(T), p = 1, \infty$.
\end{thm}

\begin{proof}
For $p = 1$, as $S$ is a positive operator,
\begin{eqnarray*}
\| Sf \|_{T,1} = T|Sf|  \le TS|f| = T|f| = \|f \|_{T,1}.
\end{eqnarray*}

For $p = \infty$,  if $|f|  \le g \in R(T)^+$, then $g=Tg$ and from the positivity of $S$ we have $|Sf|\le S|f|$ so $$|Sf|\le S|f|  \le Sg = STg = Tg= g.$$ Thus
\begin{eqnarray*}
\{ g \in R(T)^+ \, | \, |f|  \le g\} \subset \{g \in R(T)^+ \, | \, |Sf|  \le g\},
\end{eqnarray*}
from which it follows that $\|Sf \|_{T,\infty}  \le \|f \|_{T,\infty}$.
 \qed
\end{proof}

\section{Mixing inequalities}\label{section-inequality}

In this section we consider conditional versions of strong mixing ($\alpha$-mixing) and uniform mixing ($\varphi$-mixing) in the Riesz space setting. At the core of mixing processes is the family of inequalities generally termed the mixing inequalities, see \cite{bill, edgar-s, ibragimov, mix3, serfling} for the classical mixing inequalities. 
We begin by giving conditional definitions of strong and uniform mixing in the measure 
theoretic setting. These conditional definitions of mixing admit direct generalizations
to Riesz spaces with conditional expectation operators. 
We conclude with a conditional mixing inequality for conditionally strong mixing processes in Riesz spaces. This yields, directly, conditional mixing inequalities for Riesz space 
conditionally uniform mixing processes and conditional mixing inequalities for
conditionally strong and conditionally uniform mixing processes in the measure space setting.

In the classical measure theoretic setting, the strong mixing coefficient is defined as follows.
Let $(\Omega, \mathcal{F}, \mu)$ be a probability space and $\mathcal{A}$ and 
$\mathcal{B}$ be sub-$\sigma$-algebras of $\mathcal{F}$.
The strong mixing coefficient
between $\mathcal{A}$ and $\mathcal{B}$ is 
\begin{eqnarray}\label{baw-mixing-1}
\alpha(\mathcal{A},\mathcal{B}) = \sup\{|\mu(A \cap B)-\mu(A)\mu(B)| \,|\, A \in \mathcal{A}, B \in \mathcal{B}\}.
\end{eqnarray}

In place of the expectation, we could condition on 
 a sub-$\sigma$-algebra, say $\mathcal{C}$, of $\mathcal{A}\cap\mathcal{B}$ which
 would result in the $\mathcal{C}$-conditioned strong mixing coefficient
\begin{eqnarray}\label{conditioned}
\alpha_\mathcal{C}(\mathcal{A},\mathcal{B}) 
=\sup\{|\mathbb{E}[\mathbb{I}_A\mathbb{I}_B|\mathcal{C}]
-\mathbb{E}[\mathbb{I}_A|\mathcal{C}]\mathbb{E}[\mathbb{I}_B|\mathcal{C}]\,|\,
A\in\mathcal{A},B\in\mathcal{B}].
\end{eqnarray}

\begin{defs}
Let $E$ be a Dedekind complete Riesz space with weak order unit, say $e$, and conditional 
expectation operator, $T$, with $Te=e$.
If $U$ is a conditional expectation operators on $E$, with 
$TU=T=UT$, then we say that $U$ is compatible with $T$.
If $U$ is a conditional expectation on $E$ compatible with $T$ then
we denote by $\mathcal{B}(U)$ the set of band projections $P$ on $E$ with 
$Pe\in R(U)$.
\end{defs}

In light of (\ref{conditioned}), we define the strong mixing coefficient in a Riesz space with conditional expectation operator as follows.

\begin{defs}
Let $E$ be a Dedekind complete Riesz space with weak order unit, say $e$, and conditional 
expectation operator, $T$, with $Te=e$.
We define the $T$-conditional strong mixing coefficient with respect to the conditional expectation operators  $U$ and $V$ on $E$ compatible with $T$, by
$$\alpha_T(U,V):=\sup\{|TPQe-TPe\cdot TQe|\,|\,P\in \mathcal{B}(U), Q\in \mathcal{B}(V)\}.$$
\end{defs}

We can now give bounds for the $T$-conditional strong mixing coefficient in terms
of the $T$-conditional norm.

\begin{thm}
 \label{riesz}
Let $E$ be a  $T$-universally complete Riesz space, $E=L^1(T)$, where $T$ is a conditional 
expectation operator on $E$ where $E$ has a weak order unit, say $e$, with $Te=e$.
Let $U$ and $V$ be conditional expectation operators on $E$ compatible with $T$, then
\begin{eqnarray*}
\alpha_T(U,V) 
\le \sup_{Q\in \mathcal{B}(V)}\|UQe-TQe\|_{T,1}
\le 2\alpha_T(U,V).
\end{eqnarray*}
\end{thm}

\proof
Let $P\in \mathcal{B}(U)$ and $Q\in \mathcal{B}(V)$ then as $T$ is an averaging operator
$TPe\cdot TQe=T(Pe\cdot TQe)$.
The $f$-algebra structure gives that $Pe\cdot Qe=PQe$ and hence
\begin{eqnarray*}
T(Pe\cdot Qe)-TPe\cdot TQe
=T[Pe\cdot (Qe-TQe)]=T[P(Qe-TQe)].
\end{eqnarray*}
From the And\^o-Douglas-Radon-Nikod\`ym Theorem, see \cite{ando-douglas}, it follows that
\begin{eqnarray}\label{whatever}
T[P(Qe-TQe)]=T[PU(Qe-TQe)]
\end{eqnarray}
which is maximized, over $P\in \mathcal{B}(U)$, when 
$P=P_+$ is the band projection onto the band generated by
$[U(Qe-TQe)]^+$, in which case
\begin{eqnarray}
\label{new-a}
TP_+U(Qe-TQe)=T[U(Qe-TQe)]^+.
\end{eqnarray}
Hence
$$\sup\{TPQe-TPe\cdot TQe\,|\,P\in \mathcal{B}(U)\}=T[U(Qe-TQe)]^+.$$
We note that (\ref{whatever}) is minimized, over $P\in \mathcal{B}(U)$, when
$P=P_-$ is the band projection onto the band generated by
$[U(Qe-TQe)]^-$, in which case
\begin{eqnarray}
\label{new-b}
TP_-U(Qe-TQe)=-T[U(Qe-TQe)]^-.
\end{eqnarray}
Hence
$$\sup\{-TPQe+TPe\cdot TQe\,|\,P\in \mathcal{B}(U)\}=T[U(Qe-TQe)]^-.$$
Now
\begin{align*}
&\sup\{|TPQe-TPe\cdot TQe|\,|\,P\in \mathcal{B}(U)\}\\
&=\sup\{(TPQe-TPe\cdot TQe)\vee (-TPQe+TPe\cdot TQe) \,|\,P\in \mathcal{B}(U)\}\\
&\le\sup\{(TP_1Qe-TP_1e\cdot TQe)\vee (-TP_2Qe+TP_2e\cdot TQe) \,|\,P_1,P_2\in \mathcal{B}(U)\}\\
&=\sup\{(TPQe-TPe\cdot TQe) \,|\,P\in \mathcal{B}(U)\}\vee \sup\{(-TPQe+TPe\cdot TQe) \,|\,P\in \mathcal{B}(U)\}
\end{align*}
and applying (\ref{new-a}) to each of the expressions in the last line of the above gives
\begin{eqnarray*}
\sup\{|TPQe-TPe\cdot TQe|\,|\,P\in \mathcal{B}(U)\}
=T([U(Qe-TQe)]^+)\vee T([U(Qe-TQe)]^-).
\end{eqnarray*}
From the linearity of the expectation and conditional expectation operators,
$$U[(I-Q)e-T(I-Q)e]=-U[Qe-TQe].$$
Hence
$$(U[(I-Q)e-T(I-Q)e])^+=[U(Qe-TQe)]^-$$
and thus
\begin{align}
&\sup\{|TPQe-TPe\cdot TQe|\,|\,P\in \mathcal{B}(U)\}\nonumber\\
&\le
T[U(Qe-TQe)]^+\vee T(U[(I-Q)e-T(I-Q)e])^+.\label{new-z}
\end{align}
Since $Q\in \mathcal{B}(V)$ implies  $I-Q\in \mathcal{B}(V)$, 
the first inequality of the Theorem follows from taking the supremum of (\ref{new-z})
over $V\in \mathcal{B}(V)$.

Combining (\ref{whatever}), (\ref{new-a}) and (\ref{new-b}), we have
\begin{eqnarray*}
\|UQe-TQe\|_{T,1}&=&T|U(Qe-TQe)|\\
&=&T(U(Qe-TQe))^++T(U(Qe-TQe))^-\\
&=&TP_+U(Qe-TQe)-TP_-U(Qe-TQe).
\end{eqnarray*}
As $U$ is an averaging operator and $P_-,P_+\in \mathcal{B}(U)$ we have
\begin{eqnarray*}
TP_\pm U(Qe-TQe)=TUP_\pm (Qe-TQe)=TP_\pm(Qe-TQe).
\end{eqnarray*}
The positivity of $T$ and the definition of $\alpha_T(U,V)$
give
\begin{eqnarray*}
TP_\pm (Qe-TQe)\le T|P_\pm (Qe-TQe)|\le \alpha_T(U,V).
\end{eqnarray*}
Combining the above gives
\begin{eqnarray*}
\|UQe-TQe\|_{T,1}
&\le&2\alpha_T(U,V)
\end{eqnarray*}
which proves the second inequality of the Theorem.
\qed

It should be noted here that the product $TPe\cdot TQe$
exists in $E$, as shown in the previous section. 

We now consider the uniform mixing coefficient.
In the measure theoretic setting the uniform mixing coefficient is defined as follows.
Let $(\Omega, \mathcal{F}, \mu)$ be a probability space and $\mathcal{A}$ and 
$\mathcal{B}$ be sub-$\sigma$-algebras of $\mathcal{F}$.
The uniform mixing coefficient
between $\mathcal{A}$ and $\mathcal{B}$ is 
\begin{eqnarray}\label{baw-mixing-2}
\varphi(\mathcal{A},\mathcal{B}) = \sup\{|\mu(B|A)-\mu(B)| \,|\, A \in \mathcal{A}, B \in \mathcal{B},\mu(A)>0\}.
\end{eqnarray}
As with the strong mixing coefficient, the uniform mixing coefficient has an interesting
formulation in terms of $L^p$ norms.

\begin{lem}
 \label{classical-uniform}
Let $(\Omega, \mathcal{F}, \mu)$ be a probability space and $\mathcal{A}$ and 
$\mathcal{B}$ be sub-$\sigma$-algebras of $\mathcal{F}$, then
\begin{eqnarray*}
\varphi(\mathcal{A},\mathcal{B}) 
=\sup_{B\in \mathcal{B}}\|\mathbb{E}[\mathbb{I}_B-\mathbb{E}[\mathbb{I}_B]|\mathcal{A}]\|_\infty.
\end{eqnarray*}
\end{lem}

\proof
We begin by observing that, since $\mathbb{E}[\mathbb{I}_B|\mathcal{A}]-\mathbb{E}[\mathbb{I}_B]$ is $\mathcal{A}$-measurable,
$$\|\mathbb{E}[\mathbb{I}_B|\mathcal{A}]-\mathbb{E}[\mathbb{I}_B]\|_\infty
=\sup\left\{\left. \frac{|\mathbb{E}[\mathbb{I}_A(\mathbb{E}[\mathbb{I}_B|\mathcal{A}]-\mathbb{E}[\mathbb{I}_B])]|}{\mu(A)} \right| A\in\mathcal{A}, \mu(A)>0\right\}.$$
Now for $A\in\mathcal{A}$ with $\mu(A)>0$ we have
$$\frac{\mathbb{E}[\mathbb{I}_A(\mathbb{E}[\mathbb{I}_B|\mathcal{A}]-\mathbb{E}[\mathbb{I}_B])]}{\mu(A)}
=\frac{\mathbb{E}[\mathbb{I}_A\mathbb{I}_B]-\mathbb{E}[\mathbb{I}_A]\mathbb{E}[\mathbb{I}_B]}{\mu(A)}=\mu(B|A)-\mu(B),$$
from which the Lemma follows.
\qed

The above Lemma leads naturally to conditional and Riesz space variants of the 
uniform mixing coefficient.

\begin{defs}\label{def-uniform}
Let $E$ be a Dedekind complete Riesz space with weak order unit, say $e$,
 $T$ a conditional 
expectation operator on $E$ and $E$ having weak order unit say $e$ with $Te=e$.
Let $U$ and $V$ be conditional expectation operators on $E$ compatible with $T$,
 then
\begin{eqnarray*}
\varphi_T(U,V) 
=\sup_{Q\in \mathcal{B}(V)}\|UQe-TQe\|_{T,\infty}.
\end{eqnarray*}
\end{defs}

Combining Corollary \ref{baw-norm-inequality} with Theorem \ref{riesz} and Definition \ref{def-uniform} we have the following 
theorem.
  
\begin{thm}
 \label{mixing coefficients inequality}
Let $E$ be a  $T$-universally complete Riesz space, $E=L^1(T)$, where $T$ is a conditional 
expectation operator on $E$ where $E$ has a weak order unit, say $e$, with $Te=e$.
Let $U$ and $V$ be conditional expectation operators on $E$ compatible with $T$,
then
\begin{eqnarray*}
 \alpha_T(U,V)\le\varphi_T(U,V).  
\end{eqnarray*}
\end{thm}

The mixing inequalities now give bounds on the norm of the differences between the 
composition of the conditional expectation operators, say $U$ and $V$, compatible
with the conditional expectation operator $T$.
It should be noted that 
if $U$ and $V$ are conditionally independent with respect to $T$ then
$UV=T=VU$.
The measure theoretic version was proved in \cite{mcleish}, wherein results from \cite{dvoretzky} were used.

\begin{thm}
 \label{alpha mixing inequality lemma}
Let $E$ be a $T$-universally complete Riesz space, $T$ a conditional expectation operator on $E$ and $e$ a weak order unit for $E$ with $e=Te$. Let $U$ and $V$ be conditional expectation operators on $E$ compatible with $T$, then,
 for $f \in R(V)\cap L^\infty(T)$, we have 
\begin{eqnarray*}
\|Uf - Tf \|_{T,1}  \le 4\alpha_T(U,V) \|f\|_{T,\infty}.
\end{eqnarray*}
\end{thm}

\begin{proof}
Let $g:=\|f\|_{T,\infty}\in R(T)^+$, then $f$ is in the order interval $[-g,g]$.
Hence, from \cite[Theorem 4.2]{conditional}, there are sequences 
$(f_n^\pm)\subset R(V)$ with $0\le f_n^\pm\uparrow f^\pm$ of the form
$$f_n^\pm=\sum_{i=1}^{N_n^\pm} \theta_{i,n}^\pm P_{i,n}^\pm g,$$
where the band projections $P_{i,n}^\pm$ have $P_{i,n}^\pm e \in R(V)$ and $P_{i,n}^\pm P_{j,n}^\pm =0$ for $i\ne j$, and the real numbers $\theta_{i,n}^\pm$
have $0=: \theta_{0,n}^\pm<\theta_{1,n}^\pm<\theta_{2,n}^\pm< \dots<\theta_{N_n,n}^\pm$. Also $\theta_{N_n,n}^\pm\le 1$ since $f_n^\pm\le g$
and $P_{i,n}^\pm g \in R(V)$.
Set
$$Q_{i,n}^\pm=\sum_{j=i}^{N_n^\pm}  P_{i,n}^\pm,$$
then
$$f_n^\pm=\sum_{i=1}^{N_n^\pm} \beta_{i,n}^\pm Q_{i,n}^\pm g,$$
where $0<\beta_{i,n}^\pm:=\theta_{i,n}^\pm-\theta_{i-1,n}^\pm$.
Here $\displaystyle{\sum_{i=1}^{N_n^\pm} \beta_{i,n}^\pm\le 1}$ and 
 $Q_{i,n}^\pm e \in R(V)$.
 
 Now 
 $$|UQ_{i,n}^\pm g - TQ_{i,n}^\pm g|=|U(g\cdot Q_{i,n}^\pm e) - T(g\cdot Q_{i,n}^\pm e)|.$$
 Since $U$ and $T$ are averaging operators with $g\in R(T)\subset R(U)$ it follows that
 \begin{eqnarray}\label{uni-1}
 |U(g\cdot Q_{i,n}^\pm e) - T(g\cdot Q_{i,n}^\pm e)|=|g\cdot(UQ_{i,n}^\pm e - TQ_{i,n}^\pm e)|=g\cdot|UQ_{i,n}^\pm e - TQ_{i,n}^\pm e|.
 \end{eqnarray}
 Hence
 $$T|UQ_{i,n}^\pm g - TQ_{i,n}^\pm g|=T(g\cdot|UQ_{i,n}^\pm e - TQ_{i,n}^\pm e|)=g\cdot T|UQ_{i,n}^\pm e - TQ_{i,n}^\pm e|.$$

Theorem  \ref{riesz} gives 
 $$T|UQ_{i,n}^\pm e - TQ_{i,n}^\pm e|\le 2\alpha_T(U,V)$$
 and hence
 $$T|UQ_{i,n}^\pm g - TQ_{i,n}^\pm g|\le 2\alpha_T(U,V) \cdot g.$$
 Applying the above to $f_n^\pm$ gives
 \begin{align*}
 &T|Uf_{n}^\pm - Tf_{n}^\pm|\\
 &\le \sum_{i=1}^{N_n^\pm} \beta_{i,n}^\pm T|UQ_{i,n}^\pm g - TQ_{i,n}^\pm g|\\
 &\le 2\sum_{i=1}^{N_n^\pm} \beta_{i,n}^\pm \alpha_T(U,V) \cdot g\le 2\alpha_T(U,V) \cdot g.
 \end{align*}
 Taking the order limit as $n\to\infty$ and using the order continuity of conditional 
 expectation operators gives
 $$T|Uf^\pm - Tf^\pm|\le 2\alpha_T(U,V) \cdot g.$$
 Finally
 $$T|Uf - Tf|\le T|Uf^+ - Tf^+|+T|Uf^- - Tf^-|\le 4\alpha_T(U,V) \cdot g$$
 which can be rewritten as in the statement of the theorem.
 \end{proof}
\qed

Applying Theorem \ref{alpha mixing inequality lemma} to probability spaces we have
the following corollary.
\begin{cor}
Let $(\Omega,\mathcal{F},\mu)$ be a probability space,
$\mathcal{C},\mathcal{G}, \mathcal{H}$ be sub-$\sigma$-algebras of $\mathcal{F}$ with $\mathcal{G}$ and  $\mathcal{H}$ containing $\mathcal{C}$.
For $f\in L^0(\Omega,\mathcal{H},\mu)$ with $|f|$ bounded by $g\in  L^0(\Omega,\mathcal{C},\mu)$ we have
\begin{eqnarray*}
\mathbb{E}[|\mathbb{E}[f|\mathcal{G}] - \mathbb{E}[f|\mathcal{C}]|\,|\mathcal{C}]
\le 4 \alpha_\mathcal{C}(\mathcal{G},\mathcal{H}) g.
\end{eqnarray*}
\end{cor}

Setting $f=Vg$ in Theorem \ref{alpha mixing inequality lemma} and using Theorem
\ref{jensen's inequality}, we obtain the following corollary.

\begin{cor}
 \label{general-inequality}
Let $E$ be a $T$-universally complete Riesz space with weak order unit $e = Te$, where $T$ is a conditional expectation operator on $E$. Let $U$ and $V$ be conditional expectation operators on $E$ compatible with $T$. Then for $g \in L^\infty(T)$
\begin{eqnarray*}
\|UVg - Tg \|_{T,1}  \le 4 \alpha_T(U,V) \|g\|_{T,\infty}.
\end{eqnarray*}
\end{cor}

The next theorem, see \cite{serfling} for the measure theoretic case, arises using a similar procedure as the above for theorem but we now proceed from (\ref{uni-1})
and the definition of $\varphi_T(U,V)$ as follows:
 \begin{eqnarray*}
 |UQ_{i,n}^\pm g - TQ_{i,n}^\pm g|=|g\cdot(UQ_{i,n}^\pm e - TQ_{i,n}^\pm e)|\le g\cdot\varphi_T(U,V).
 \end{eqnarray*}
Here
 $$|Uf_{n}^\pm - Tf_{n}^\pm|
 \le \sum_{i=1}^{N_n^\pm} \beta_{i,n}^\pm |UQ_{i,n}^\pm g - TQ_{i,n}^\pm g|
 \le \sum_{i=1}^{N_n^\pm} \beta_{i,n}^\pm \varphi_T(U,V) \cdot g\le \varphi_T(U,V) \cdot g.$$
Letting $n\to\infty$ gives
 $$|Uf^\pm - Tf^\pm|
\le \varphi_T(U,V) \cdot g$$
and thus 
 $$\|Uf - Tf\|_{T,\infty}
\le 2\varphi_T(U,V) \cdot \|f\|_{T,\infty}$$
from which the following theorem follows.

\begin{thm}
\label{phi-inequality}
Let $E$ be a $T$-universally complete Riesz space with weak order unit $e=Te$, where $T$ is a conditional expectation operator on $E$. Let $U$ and $V$ be conditional expectation operators on $E$ compatible with $T$. Then for $f \in L^{\infty}(T) \cap R(V)$
\begin{eqnarray*}
\|Uf - Tf\|_{T,1}  \le \|Uf - Tf\|_{T,\infty}  \le 2 \varphi_T(U,V) \|f\|_{T,\infty}.
\end{eqnarray*}
\end{thm}

Applying Theorem \ref{phi-inequality} to probability spaces we have the following corollary.
\begin{cor}
Let $(\Omega,\mathcal{F},\mu)$ be a probability space,
$\mathcal{C},\mathcal{G}, \mathcal{H}$ be sub-$\sigma$-algebras of $\mathcal{F}$ with $\mathcal{G}$ and  $\mathcal{H}$ containing $\mathcal{C}$.
For $f\in L^0(\Omega,\mathcal{H},\mu)$ with $|f|$ bounded by $g\in  L^0(\Omega,\mathcal{C},\mu)$ we have
\begin{eqnarray*}
|\mathbb{E}[f|\mathcal{G}] - \mathbb{E}[f|\mathcal{C}]| 
\le 2 \varphi_\mathcal{C}(\mathcal{G},\mathcal{H}) g.
\end{eqnarray*}
\end{cor}

\section{Mixing for $\sigma$-finite processes}

In this section we consider the simplest non-trivial application, that is,
to $\sigma$-finite processes,
hence giving a theory of conditional mixing for such processes.
In this concrete example the  spaces and operators can be clearly identified.
A consideration of $\sigma$-finite processes in the context of martingale theory can be found in
the work of Dellacherie and Meyer, \cite[Sections 39, 42 and 43]{D-M-B}. 

Let $(\Omega, \mathcal{A},\mu)$ be a $\sigma$-finite measure space, which to be interesting should have $\mu(\Omega)=\infty$, and let $(\Omega_i)_{i\in\N}$ be a $\mu$-measurable partition of $\Omega$ into sets of finite positive measure.
Let $\mathcal{A}_0$ be the sub-$\sigma$-algebra of $\mathcal{A}$ generated by $(\Omega_i)_{i\in\N}$. We take the Riesz space $E=L^\infty(\Omega, \mathcal{A},\mu)$ and
the conditional expectation operator $T=\mathbb{E}[\,\cdot\,|\mathcal{A}_0]$.
For $f\in E$ we have 
\begin{eqnarray}
Tf(\omega)=\frac{\int_{\Omega_i}f\,d\mu}{\mu(\Omega_i)},\quad\mbox{for}\quad \omega\in \Omega_i.\label{app-1}
\end{eqnarray}
Here we have that the universal completion, $E^u$, of $E$ is the space of all $\mathcal{A}$-measurable functions. The $T$-universal completion of $E$ is the space
$$\mathcal{L}^1(T)=\left\{f\in E^u\,\left|\,\int_{\Omega_i}|f|\,d\mu<\infty \mbox{ for all } i\in\N\right\},\right.$$
which is characterized by $f|_{\Omega_i}\in L^1(\Omega, \mathcal{A},\mu)$, for each $i\in \N$. Here $T$ 
can be extended to an $\mathcal{L}^1(T)$ conditional expectation operator as per (\ref{app-1}).
We note that $E$ has weak order unit $e=1$, the function identically $1$ on $\Omega$, which again is a weak order unit for  $\mathcal{L}^1(T)$, but is not in $L^1(\Omega, \mathcal{A},\mu)$. The range of the generalized conditional expectation operator $T$ is
$$R(T)=\left\{f\in E^u\,\left|\,f \mbox{ a.e. constant on } \Omega_i, i\in\N\right\},\right.$$
which is an $f$-algebra. The last of the spaces to be considered is
$$\mathcal{L}^\infty(T)=\left\{f\in E^u\,\left|\,f \mbox{ essentially bounded on } \Omega_i \mbox{ for each } i\in\N\right\}.\right.$$
The vector norms on $\mathcal{L}^1(T)$ and $\mathcal{L}^\infty(T)$ are
\begin{eqnarray}
 \|f\|_1(\omega)&=&T|f|(\omega)=\frac{\int_{\Omega_i}|f|\,d\mu}{\mu(\Omega_i)},\quad\mbox{for}\quad \omega\in \Omega_i, f\in \mathcal{L}^1(T),\label{app-2}\\
 \|f\|_\infty(\omega)&=&\mbox{ess sup}_{\Omega_i} |f|,\quad\mbox{for}\quad \omega\in \Omega_i, f\in \mathcal{L}^\infty(T).\label{app-3}
\end{eqnarray}
We note that $L^1(\Omega, \mathcal{A},\mu)\subsetneq \mathcal{L}^1(T)$,
$L^\infty(\Omega, \mathcal{A},\mu)\subsetneq \mathcal{L}^\infty(T)$,
$\mathcal{L}^\infty(T)\subset \mathcal{L}^1(T)$ while $L^\infty(\Omega, \mathcal{A},\mu)
\not\subset L^1(\Omega, \mathcal{A},\mu)$.

Let $\mathcal{C}$ and $\mathcal{D}$ be sub-$\sigma$-algebras of $\mathcal{A}$ which contain
$\mathcal{A}_0$. The $\alpha$-mixing coefficient of $\mathcal{C}$ and $\mathcal{D}$ conditioned on $\mathcal{A}_0$ (which in measure theoretic terms could be denote
$\alpha_{\mathcal{A}_0}(\mathcal{C},\mathcal{D})$  is $\alpha_T(U,V)$.
Here $U$ and $V$ are the restrictions to $\mathcal{L}^1(T)$ of the extensions to 
$\mathcal{L}^1(U)$ and $\mathcal{L}^1(V)$ respectively of the conditional expectation operators $U$ and $V$ on $E$ conditioning 
with respect to the $\sigma$-algebras $\mathcal{C}$ and $\mathcal{D}$. In this example case these operators can be given explicitly by
\begin{eqnarray}
U(f)&=&\sum_{i=1}^\infty \mathbb{E}_i[f\mathbb{I}_{\Omega_i}|\mathcal{C}],\\
V(f)&=&\sum_{i=1}^\infty \mathbb{E}_i[f\mathbb{I}_{\Omega_i}|\mathcal{D}],
\end{eqnarray}
for $f\in  \mathcal{L}^1(T)$. 
Here the conditional expectation 
$\mathbb{E}_i[f\mathbb{I}_{\Omega_i}|\mathcal{C}]=\mathbb{E}_i[f|\mathcal{C}]$ is the conditional
expectation on $\Omega_i$ of $f|_{\Omega_i}$ with respect to the
probability measure $\mu_i(A):=\frac{\mu(A\cap\Omega_i)}{\mu(\Omega_i)}$ 
and the $\sigma$-algebra $\{C\cap \Omega_i|C\in \mathcal{C}\}$, and similarly
for $\mathcal{C}$ replaced by $\mathcal{D}$.
Since the structure of the example is extremely simple, an explicit computation can be carried out to give
\begin{eqnarray*}
\alpha_T(U,V)&=&\alpha_{\mathcal{A}_0}(\mathcal{C},\mathcal{D})\\
&=&\sum_{i=1}^\infty\mathbb{I}_{\Omega_i}\sup\left.\left\{\left|\mu_i(C\cap D)
-\mu_i(C)\mu_i(D)\right|\,\right|\,C\in \mathcal{C}, D\in\mathcal{D}\right\}\\
&=&\sum_{i=1}^\infty
\alpha_i(\mathcal{C},\mathcal{D})\mathbb{I}_{\Omega_i},
\end{eqnarray*}
where $\alpha_i(\mathcal{C},\mathcal{D})$ is the $\alpha$-mixing coefficient of
$\sigma$-algebras $\mathcal{C}$ and $\mathcal{D}$ with respect to the 
probability measure $\mu_i$.
Corollary \ref{general-inequality} gives that if $g$ is $\mu$-measurable and
essential bounded on each $\Omega_i, i\in \N$, then
\begin{eqnarray*}
\|UVg - Tg \|_{T,1}  \le 4 \alpha_T(U,V) \|g\|_{T,\infty},
\end{eqnarray*}
which in this example case can be written as
\begin{eqnarray*}
{\frac{1}{\mu(\Omega_i)}\int_{\Omega_i}\left|\mathbb{E}_i[\mathbb{E}_i[g|\mathcal{D}]|\mathcal{C}] - \frac{1}{\mu(\Omega_i)}\int_{\Omega_i}g\,d\mu \right| d\mu}\,  \le\, 4\alpha_i(\mathcal{C},\mathcal{D})  \mbox{ess sup}_{\Omega_i} |g| 
\end{eqnarray*}
for each $i\in\N$.
The conditional uniform mixing coefficient is given by
\begin{eqnarray*}
\varphi_T(U,V) &=&\varphi_{\mathcal{A}_0}(\mathcal{C},\mathcal{D})\\
&=&\sum_{i=1}^\infty\mathbb{I}_{\Omega_i}\sup_{D\in \mathcal{D}}\mbox{ess sup}_{\Omega_i} \left|\mathbb{E}_i[\mathbb{I}_{D}|\mathcal{C}]-\mu_i(D)\right|\\
&=&\sum_{i=1}^\infty
\varphi_i(\mathcal{C},\mathcal{D})\mathbb{I}_{\Omega_i},
\end{eqnarray*}
where $\varphi_i(\mathcal{C},\mathcal{D})$ is the $\varphi$-mixing coefficient of 
$\mathcal{C}$ and $\mathcal{D}$ relative to the probability measure $\mu_i$.
For $g$ as above, Theorem \ref{phi-inequality} gives
\begin{eqnarray*}
\|UVg - Tg\|_{T,\infty}  \le 2 \varphi_T(U,V) \|g\|_{T,\infty},
\end{eqnarray*}
which in the special case under consideration yields
\begin{eqnarray*}
\left|\mathbb{E}_i[\mathbb{E}_i[g|\mathcal{D}]|\mathcal{C}] - \frac{1}{\mu(\Omega_i)}\int_{\Omega_i}g\,d\mu \right| 
\le 2\varphi_i(\mathcal{C},\mathcal{D}) \mbox{ ess sup }_{\Omega_i} |g| 
\end{eqnarray*}
a.e. on $\Omega_i$, for each $i\in\N$.

It should 
be noted that the work presented here also applies to processes where the random variables are Riesz space valued, say $L^p$, and the conditional expectation, $T$, is generated by an arbitrary sub-$\sigma$-algebra of $\mathcal{A}$. In this case we obtain a generalization of 
mixing to the context of vector measure.


\end{document}